\documentclass[leqno]{amsart}
\usepackage{amsmath,amssymb,latexsym,amsfonts,mathrsfs}
\usepackage{amscd}
\usepackage{amsthm}
\usepackage{bm}
\usepackage[all]{xy}
\usepackage{verbatim}
\usepackage{enumerate}
\allowdisplaybreaks

\newtheorem{mytheorem}{Theorem}[section]
\newtheorem{prop}[mytheorem]{Proposition}
\newtheorem{lem}[mytheorem]{Lemma}
\newtheorem{cor}[mytheorem]{Corollary}

\theoremstyle{definition}
\newtheorem{thm}{Theorem}[section]
\newtheorem{rem}[mytheorem]{{Remark}}

\def\dbar{\mbox{\setbox0=\hbox{$d$}$d$\kern-.75\wd0\vbox{%
\hrule height.1ex width.80\wd0\kern1.2ex}}}
\newcommand{\R}{\mathbb{R}}

\newcommand{\N}{\mathbb{N}}
\newcommand{\Z}{\mathbb{Z}}

\newcommand{\CAL}[1]{{\mathcal #1}}

\newcommand{\Sh}{\CAL{S}}
\newcommand{\ep}{\varepsilon}

\newcommand{\la}{\lambda}

\renewcommand{\labelenumi}{(\roman{enumi})}
\numberwithin{equation}{section}

\newtheorem{ass}[mytheorem]{Assumption}

\newtheorem{definition}[mytheorem]{Definition}

\title[Wave front set of solutions to magnetic Schr\"{o}dinger equations]
{Wave front set of solutions to Schr\"{o}dinger equations with time-dependent magnetic fields}
\date{}
\author[F.Abe]{Fumihito Abe}
\address[F.Abe]{J Institute Co., Ltd., 2-1-1 Shibuya,
Shibuya-ku, Tokyo, 150-0002, Japan}
\email[F.Abe]{fumihito.abe@jprep.jp}
\author[R.Muramatsu]{Ryo Muramatsu}
\address[R.Muramatsu]{Department of Mathematics, Faculty of Science Division I, Tokyo University of Science, Kagurazaka 1-3, Shinjuku-ku, Tokyo, 162-8601, Japan}
\email[R.Muramatsu]{rmuramatsu@rs.tus.ac.jp}

\begin{document}
\maketitle

\begin{abstract}
In this paper, we determine the wave front set of solutions to the Schr\"{o}dinger equation with time-dependent magnetic fields.
We considered time-dependent and `not so small' magnetic fields through the method using the wave packet transform established by K. Kato, M. Kobayashi and S. Ito \cite{KI, KKI2, KKI3}.
Furthermore,
we checked that the fundamental solution of the Schr\"odinger equation in a spatially decaying magnetic field has no singularities as a consequence of our result.
\end{abstract}

\section{Introduction}
\par
The aim of this paper is to determine the wave front sets of solutions of the following initial value problem for Schr\"{o}dinger equation with time-dependent magnetic fields;
\begin{align}\label{eq1}
\begin{cases}
i \partial _t u(t,x) + \frac{1}{2}(\nabla-\bm{a}(t,x))^2u(t,x)=0,\quad (t, x)\in \R\times \R^n, \\
u(0,x) = u_0 (x),\quad x\in \R^n,
\end{cases}
\end{align}
and to give some applications, 
where
$n\geq 2$, $i=\sqrt{-1}$, $u(t,x)$ is a complex-valued unknown function of $(t,x)\in \R\times \R^n$, $u_0(x)$ is a complex-valued given function of $x\in{\mathbb R}^n$, $\partial_t u=\partial u/\partial t$, $\partial_{x_j} u = \partial u/\partial x_j\ (j=1,\ldots ,n)$ and $\nabla=(\partial_{x_1},\ldots,\partial_{x_n})$.
\par
Throughout this paper, we assume that $\bm{a}(t,x)$ satisfies the following Assumption.
\begin{ass}\label{A2}
For $j=1,\ldots, n$, $j$-th component $a_j(t, x):\mathbb{R}\times\mathbb{R}^n\rightarrow \mathbb{R}$ of $\bm{a}(t,x)$  is a $C^{\infty}$-function with respect to $(t,x)$ and satisfies that there exists $\rho<1$ such that for any $\alpha \in \mathbb{Z}^n_{+}$, there exists a constant $C_{\alpha}>0$ satisfying
\begin{align}\label{GC}
\max_{1\leq j \leq n}|\partial_x^{\alpha}a_j(t, x)|\leq C_{\alpha} \langle x\rangle^{\rho - |\alpha|},\quad (t,x)\in\mathbb{R}\times\mathbb{R}^n
\end{align}
where, $\langle x\rangle=(1+|x|^2)^{1/2}$.
\end{ass}
We note the magnetic field $B(t,x)$ corresponds to the $n\times n$-matrix $(\partial_{x_j}a_k(t,x)-\partial_{x_k}a_j(t,x))_{1\leq j,k\leq n}$.
\par
To state our results, we define some notations.
Let $\CAL{S}(\R^n)$ be the space of rapidly decreasing smooth functions on $\R^n$ and $\CAL{S}'(\R^n)$ be its dual space.
\begin{definition}[Wave packet transform]
Let $\varphi(x)\in \CAL{S}(\mathbb{R}^n)\setminus \{0\}$ and $f \in \CAL{S} '(\mathbb{R}^n)$. The wave packet transform $W_{\varphi}f$ of $f$ with the basic wave packet $\varphi$ is defined by
\begin{align*}
W_{\varphi}f(x, \xi)\ =\ \int_{{\mathbb R}^n} \overline{\varphi(y-x)}f(y)e^{-iy\cdot \xi}dy,\quad (x, \xi) \in{\mathbb R}^n\times{\mathbb R}^n.
\end{align*}
\end{definition}
\par
For $\varphi(x) \in \mathcal{S}\left(\mathbb{R}^{n}\right) \backslash\{0\}$ and $0<b<1$, we put $\varphi_{\lambda}(x)=\lambda^{n b / 2} \varphi\left(\lambda^{b} x\right)$.
We write $\widehat{f}(\xi)=\CAL{F}[f](\xi)=(2\pi)^{-n/2}\int_{\R^n}f(x)e^{-ix\xi}dx$ for $f\in\CAL{S}'$.
For $\xi_{0} \in \mathbb{R}^{n} \backslash\{0\}$, a neighborhood $\Gamma$ of $\xi_{0}$ in $\mathbb{R}^{n}$ is called a conic neighborhood of $\xi_{0}$ if $\xi \in \Gamma$ and $\alpha>0$ imply $\alpha \xi \in \Gamma$.
\begin{definition}[wave front set]
    For $f \in \mathcal{S}^{\prime}\left(\mathbb{R}^{n}\right)$, we say $\left(x_{0}, \xi_{0}\right) \notin W F(f)$ if there exist a function $\chi(x)$ in $C_{0}^{\infty}\left(\mathbb{R}^{n}\right)$ with $\chi\left(x_{0}\right) \equiv 1$ near $x_{0}$ and a conic neighborhood $\Gamma$ of $\xi_{0}$ such that for all $N \in \mathbb{N}$ there exists a positive constant $C_{N}$ satisfying
    $$
    |\widehat{\chi f}(\xi)| \leq C_{N}(1+|\xi|)^{-N}
    $$
    for all $\xi \in \Gamma$.
    \end{definition}    

For $\lambda \geq 1$, $t_{0} \in \mathbb{R}$ and $(x, \xi) \in \mathbb{R}^{n} \times \mathbb{R}^{n}$, $x(s)=x\left(s ; t_{0}, x, \lambda \xi\right)$ and $\xi(s)=\xi\left(s ; t_{0}, x, \lambda \xi\right)$ denote solutions to
\begin{align}\label{ce}
\begin{cases}
\dot{x}(s)=(\nabla_{\xi}h)(s, x(s), \xi(s)), & x(t_0) = x,\\
\dot{\xi}(s)=-(\nabla_{x}h)(s, x(s), \xi(s)), & \xi(t_0) = \la\xi,
\end{cases}
\end{align}
where $h(t,x,\xi)=\frac{1}{2}\left|\xi-\bm{a}(t,x)\right|^2$ and $\dot{X}(s)=\frac{d}{ds}X(s)$.
\par
Our purpose in this study is to characterize the wave front sets of  solutions of the initial value problem \eqref{eq1} by the initial data.
The following theorem is our main result.
\begin{thm}\label{mt1}
Suppose $\bm{a}(t,x)$ satisfies Assumption~\ref{A2}. 
Let $b=\min\left(\frac{1}{8}, \frac{1-\rho}{8}\right)$.
Then for any $u_0(x)\in L^2(\R^n)$, $\varphi_0(x)\in \CAL{S}(\mathbb{R}^n)\setminus \{0\}$, the solution $u(t,x)\in C(\R;L^2(\R^n))$ of 
\eqref{eq1} and $t_0\in\R$, the following conditions are equivalent:
\renewcommand{\labelenumi}{(\roman{enumi})}
\begin{enumerate}
    \item $(x_0,\xi_0)\notin WF(u(t_0))$.
    \item There exist a neighborhood $K$ of $x_0$ and a conic neighborhood $\Gamma$ of $\xi_0$ such that for all $N\in\N$, $a\geq1$ and $\varphi_0(x)\in\Sh(\R^n)\setminus\{0\}$, there exists a constant $C_{N,a,\varphi_0}>0$ satisfying the inequality
    \begin{equation}\label{main}
      \left|W_{\varphi_{\la}^{(-t_0)}}u_0(x(0;t_0,x,\la\xi),\xi(0;t_0,x,\la\xi))\right|\leq C_{N,a,\varphi_0}\la^{-N}
    \end{equation}
    for all $\la\geq1$, $x\in K$ and $\xi\in\Gamma_a:=\left\{\xi\in\Gamma\middle|a^{-1}\leq |\xi|\leq a\right\}$, where $W_{\varphi_{\la}^{(-t_0)}}u_0(x,\xi)$ denotes the wave packet transform of $u_0$ with a basic wave packet $\varphi_{\la}^{(-t_0)}(x)$ and, for $t\in\R$, $\varphi_{\lambda}^{(t)}(x)$ stands for $\varphi_{\lambda}^{(t)}(x)=e^{it\Delta/2}\varphi_{0,\lambda}(x)$.
    \item There exist $\varphi_0(x)\in\Sh(\R^n)\setminus\{0\}$, a neighborhood $K$ of $x_0$ and a conic neighborhood $\Gamma$ of $\xi_0$ such that for all $N\in\N$, $a\geq1$, there exists a constant $C_{N,a,\varphi_0}>0$ satisfying the inequality \eqref{main} for all $\la\geq1$, $x\in K$ and $\xi\in\Gamma_a$.
\end{enumerate}
\end{thm}
Invoking the result and method of Kato--Ito~\cite{KI}, we can extend our result to the case with a scalar potential perturbation term:
 \begin{cor}
	Let $\bm{a}(t,x)$ satisfy Assumption~\ref{A2}, $u(t,x)\in C(\R;L^2(\R^n))$ be the solution of the equation
	\begin{align}\label{eq2}
\begin{cases}
i \partial _t u(t,x) + \frac{1}{2}(\nabla-\bm{a}(t,x))^2u(t,x)=V(t,x)u(t,x),\quad (t, x)\in \R\times \R^n, \\
u(0,x) = u_0 (x),\quad x\in \R^n,
\end{cases}
\end{align}
with $u_0(x)\in L^2(\R^n)$ and $V(t,x)\in C^{\infty}(\R\times\R^n)$ satisfy that there exists $\mu<2$ such that for any $\alpha \in \mathbb{Z}^n_{+}$, there exists a constant $C_{\alpha}>0$ satisfying
\begin{align*}
|\partial_x^{\alpha}V(t, x)|\leq C_{\alpha} \langle x\rangle^{\mu - |\alpha|},\quad (t,x)\in\mathbb{R}\times\mathbb{R}^n.
\end{align*}
Then the equivalence of all conditions of Theorem~\ref{mt1} still holds.
 \end{cor}
 
As a main application of our main result, we conclude that  the fundamental solution of the \eqref{eq1} with the magnetic field $B(t,x)$ decaying
at infinity has no singularities except at the initial time.
\begin{cor}[Absence of singularities of the fundamental solution]\label{cor1}\hfill\\
Let $H(t)=\frac{1}{2}(\nabla-i\bm{a}(t,x))^2$ with $\bm{a}(t,x)$ satisfying Assumption \ref{A2} and $E(t,x,0)$ denote the fundamental solution of \eqref{eq1}, namely it is the distributional solution of \eqref{eq1} with initial data $\delta_0(x)$.
Then
$WF(E(t_0,\cdot,0))=\emptyset$ for $t_0\neq0.$
\end{cor}
The variation of this proposition has already been shown in the paper of W. Craig, T. Kappeler and W. Strauss \cite[Theorem 1.9]{CKS} or as a corollary of K. Yajima's work \cite{Y0} of constructing the fundamental solution (see also \cite{Y1, Y2} for the case of scalar potentials). 
Our contribution is that we extended these results to cases where $\bm{a}(t,x)$ grows spatially and depends on $t$ without an exact representation of the fundamental solution and restricting the time interval sufficiently short.
\par
The wave front set was originally introduced for detecting singularities by directions, i.e.`\textit{microlocally}', by L. H\"ormander (see \cite{H2}). It has been used to  investigate phenomena of propagation of singularities for several pseudo-differential operators including hyperbolic operators of principal type initially.
\par
In last decades, various authors examined the propagation of singularities for Schr\"odinger operators, in which there are two aspects of the picture concurrently: the recurrence of singularities caused by trapping of classical orbits and the microlocal smoothing effect due to nontrapping situation.
The former is investigated in such as \cite{L}, \cite{Z}, \cite{KRY} and \cite{W} and the latter is in such as \cite{CKS}, \cite{Do2} and so on.
In particular, S. Doi in \cite{Do2} discovered that the smoothing effect of the Schr\"odinger operator on a $C^{\infty}$ manifold propagates along the geodesic flow.
\par
S. Nakamura has determined the wave front set of various types of Schr\"odinger equations, which include not only the free equation but also with variable coefficients, harmonic oscillators and their perturbed equations, in successive works \cite{N2}, \cite{MN} and \cite{N3}.
\par
There has been much research on the wave front set of the Schr\"odinger operator on the manifold, whose  geometric situation changes.
However, except for S. Mao's works, it seems there are very few results for the operator in magnetic fields, although magnetic fileds also give a large change for the geometric property of the classical orbit.
There are successive works \cite{M1} and \cite{M2}, in which S. Mao has studied the wave front set for the Schr\"odinger operator in (perturbed) magnetic fields.
His works stand on the point of view what there is recurrence of singularities for the case of constant magnetic fields and it is not affected by `small' perturbations.
As so far, there are no results for the case of general magnetic vector potentials such as of growing spatially and depending on time. 
We treated these type magnetic potentials even depending on time in this paper. 
\par
One of the novelties of our results is that we addressed this problem using the method developed by K. Kato, M. Kobayashi and S. Ito.
Their method adopts the wave packet transform introduced originally by A. C\'ordoba and C. Fefferman in \cite{CF}.
Inspired by \=Okaji's works \cite{O1} and \cite{O2}, they characterized the wave front set in terms of the wave packet transform in \cite{KKI3}. 
Then they combined this result and the representation of the solution of the Schr\"odinger equation also obtained by them (see \cite{KKI2} and  \cite{KKI4}), and  characterized the wave front set of the solution of the Schr\"odinger equation with harmonic oscillators in \cite{KKI1}, and  then with time-dependent sub-quadratic scalar potentials in \cite{KI}.
Later their results have been refined for perturbed harmonic oscillators \cite{KI2}.
This method has the advantage that allows us to treat time-dependent potentials and `long-range' potentials.
This is why we adopted their method.
\par
Finally, we will explain the meaning of assuming such a general assumption on the vector potential.
Assumption \ref{A2} implies that the corresponding magnetic field $B(t,x)$ is decaying spatially as $|B(t,x)|\lesssim \langle x\rangle^{\rho-1}$ at infinity.
Such $B(t,x)$ appears in various situations from the Biot--Savart law thus it may have some meanings to concern about the general spatially decaying magnetic fields.
For instance, an electric circular current makes the magnetic field $B(x)=(0,0,(1+x_3^2)^{-1})$ for $x=(x_1,x_2,x_3)\in\R^3$ at its central axis.
Also, the electric linear current makes
the magnetic field $B(t,x)=|x|^{-2}x^{\perp}$ ($x=(x_1,x_2)\in\R^2$, $x^{\perp}=(x_2,-x_1)$) on the plane.
We can regard that this magnetic field is also treated in this paper with the regularizing modification at the origin. 
\par
This paper is organized as follows. In Section \ref{sec2}, we introduce terminologies and preliminaries we need. Especially, we give the representation of the solution of \eqref{eq1} by wave packet transform and introduce some lemmas on characteristics corresponding to \eqref{eq1} in Section \ref{sec2}.
In Section \ref{sec3}, we prepare some key lemmas that play an important role in the proof of our theorem.
Then we will give a proof of Theorem \ref{mt1} in Section 4, and we will show Corollary \ref{cor1} in Section 5.
\par
\section{Preliminaries}\label{sec2}
\par
\subsection{Transform of equation \eqref{eq1}}
\renewcommand{\labelenumi}{(\theenumi)}
\par
\begin{definition}[Inverse wave packet transform]
Let $\varphi \in \CAL{S}(\mathbb{R}^n)\setminus \{0\}$ and $F\in\CAL{S} '(\mathbb{R}^{2n})$, we define the adjoint operator $W_{\varphi}^{*}$ of $W_{\varphi}$ by 
\begin{align*}
W_{\varphi}^{*}[F(y, \xi)](x)\ =\ \iint_{{\mathbb R}^{2n}} \varphi(x-y)F(y, \xi)e^{ix\cdot \xi}dy\dbar\xi,\quad x\in{\mathbb R}^n,
\end{align*}
where $\dbar\xi=(2\pi)^{-n}d\xi$.

Then, for $f\in{\CAL S}'({\mathbb R}^n)$, the following inversion fomula holds (see \cite[Corollary 11.2.7]{Gr});
\begin{align}\label{if}
f(x) = \frac{1}{\|\varphi\|_{L^2}^2}
W_{\varphi}^{*}[W_{\varphi}f](x).
\end{align}
\end{definition}
For the proof of Theorem \ref{mt1}, we prepare the representation of the solution $u(t,x)$ of \eqref{eq1} by wave packet transform.
\begin{lem} 
For $t\in\mathbb{R}$ and $x, \xi\in\mathbb{R}^n$, we define $x(s)=x(s; t, x, \xi)$ and $\xi(s)=\xi(s; t, x, \xi)$ as solutions of \eqref{ce}.
 If $\bm{a}(t,x)$ satisfies Assumption \ref{A2},
Then the solution $u(t, x)$ of \eqref{eq1} satisfies the following integral equation
 \begin{align}\label{fo}
 &W_{\varphi^{(t)}}u(t, x, \xi)\\
 &=e^{-i\int_{0}^{t}\Psi(s, x(s), \xi(s))ds}\bigg(W_{\varphi_0}u_{0}(x(0;t, x, \xi), \xi(0; t, x, \xi))\notag\\
 &\quad-i\int_0^te^{i\int_{0}^{\tau}\Psi(s, x(s), \xi(s))ds}R\left[\varphi^{(t)},u\right](\tau, x(\tau; t, x, \xi), \xi(\tau; t, x, \xi))d\tau\bigg),\notag
 \end{align}
 where $\varphi^{(t)}(x)=e^{it\Delta/2}\varphi_0(x)$, $\Psi(t, x, \xi)=-h(t,x,\xi)+\nabla_xh(t,x,\xi)\cdot x+\frac{i}{2}\nabla_x\cdot\bm{a}(t,x)$,
\begin{align}\label{remainder}
&R\left[\varphi^{(t)},u\right](t,x,\xi)\equiv R_1\left[\varphi^{(t)},u\right](t,x,\xi)+R_2\left[\varphi^{(t)},u\right](t,x,\xi)\\
&=\sum_{j=1}^n\int u(t,y)e^{-iy\cdot\xi} \bigg(-\frac{1}{2}\left(r^{L,j}_{0,0}(t,x,y)\right)^2+\xi_jr^{L,j}_{0,0}(t,x,y)+\frac{i}{2}r^{L,j}_{0,1}(t,x,y)\notag\\
&\qquad\qquad+i\left(\nabla a_j(t,x)\cdot(y-x)+r^{L,j}_{0,0}(t,x,y)\right)\partial_{y_j}\bigg)
\overline{\varphi(t,y-x)}dy\notag\\
&+\sum_{j=1}^n\int u(t,y)e^{-iy\cdot\xi} \bigg(-\frac{1}{2}\left(r^{L,j}_{1,0}(t,x,y)^2+2r^{L,j}_{0,0}(t,x,y)r^{L,j}_{1,0}(t,x,y)\right)\notag\\
&\qquad\qquad+\xi_jr^{L,j}_{1,0}(t,x,y)+\frac{i}{2}r^{L,j}_{1,1}(t,x,y)+ir^{L,j}_{1,0}(t,x,y)\partial_{y_j}\bigg)
\overline{\varphi(t,y-x)}dy\notag
\end{align}
and
\begin{align*}
r_{0,0}^{L,j}(t,x,y)&=
\sum_{2\leq|\alpha|\leq L-1}\frac{\partial_{x}^{\alpha}a_j(t,x)}{\alpha!}(y-x)^{\alpha},\\
r_{1,0}^{L,j}(t,x,y)&=\sum_{|\alpha|=L}\frac{(y-x)^{\alpha}}{\alpha!}\int_0^1\partial_{x}^{\alpha}a_j\left(t,y(x,\theta)\right)(1-\theta)^{L-1}d\theta,\\
r_{0,1}^{L,j}(t,x,y)&=
\sum_{1\leq|\alpha|\leq L-2}\frac{\partial_{x}^{\alpha}\partial_{x_j}a_j(t,x)}{\alpha!}(y-x)^{\alpha},\\
r_{1,1}^{L,j}(t,x,y)&=\sum_{|\alpha|=L-1}\frac{(y-x)^{\alpha}}{\alpha!}\int_0^1\partial_{x}^{\alpha}\partial_{x_j}a_j\left(t,y(x,\theta)\right)(1-\theta)^{L-2}d\theta,\\
y(x,\theta)&=x+\theta(y-x)
\end{align*}
where $a_j$ is $j$-th component of $\bm{a}(t,x)$.
\end{lem}
\begin{proof}
Firstly we reduce \eqref{eq1} to a first-order partial differential equation in $\mathbb{R}^{2n}$ by using the wave packet transform.
Set Taylor expansions of $a_j(t,x)$ and $\partial_{x_j}a_j(t,x)$ as
\begin{align*}
a_j(t,y)&=a_j(t,x)+\nabla a_j(t,x)\cdot(y-x)+r_{0,0}^{L,j}(t,x,y)+r_{1,0}^{L,j}(t,x,y),\\
\partial_{y_j}a_j(t,y)&=\partial_{x_j}a_j(t,x)+r_{0,1}^{L,j}(t,x,y)+r_{1,1}^{L,j}(t,x,y).
\end{align*}
These and integration by parts yield that
\begin{align}\label{2}
&W_{\varphi(t,\cdot)}\left[i\partial_tu+\frac{1}{2}(\nabla-i\bm{a})^2u\right](t,x,\xi)\\
=&W_{\varphi(t,\cdot)}\left[i\partial_tu+\frac{1}{2}\Delta u-\frac{1}{2}|\bm{a}|^2u-i\left(\frac{1}{2}(\nabla\cdot \bm{a})+\bm{a}\cdot\nabla \right)u\right](t,x,\xi)\notag \\
=&\bigg(i\partial_t+i\xi\cdot\nabla_x-\frac{|\xi|^2}{2}-\frac{1}{2}|\bm{a}(t,x)|^2-\frac{i}{2}\nabla_x|\bm{a}(t,x)|^2\cdot\nabla_{\xi}\notag\\
&\qquad+\frac{1}{2}\nabla_x|\bm{a}(t,x)|^2\cdot x+\frac{i}{2}(\nabla_x\cdot \bm{a})(t,x)-i\bm{a}(t,x)\cdot\nabla_x\notag\\
&\qquad+\xi\cdot \bm{a}(t,x)+\nabla_x(\xi\cdot \bm{a}(t,x))\cdot(i\nabla_{\xi}-x)\bigg)W_{\varphi(t,\cdot)}u(t,x,\xi)\notag\\
&+W_{\left(i\partial_t+\frac{\Delta}{2}\right)\varphi(t,\cdot)}u(t,x,\xi)+R\left[\varphi^{(t)},u\right](t,x,\xi),\notag
\end{align}
where $R\left[\varphi^{(t)},u\right](t,x,\xi)$ is \eqref{remainder}
(for more details on deduction of \eqref{2}, we refer \cite{KM}, \cite{M} or \cite{KKI4}).
Taking $\varphi(t,x)=e^{it\Delta/2}\varphi_0(x)$ for  $\varphi_0(x)\!\in\!\CAL{S}(\R^n)\setminus\{0\}$ and 
combining \eqref{2}, we transform \eqref{eq1} into
\begin{align*}
\begin{cases}
(i \partial _t + i\nabla_{\xi}h(t, x, \xi)\cdot\nabla_x-i\nabla_{x}h(t, x, \xi)\cdot \nabla_{\xi}\\
\hspace{4cm}+\Psi(t,x,\xi))W_{\varphi(t, \cdot)}u(t, x, \xi)=R\left[\varphi^{(t)},u\right](t,x,\xi), \\
W_{\varphi(0, \cdot)}u(0,x,\xi) = W_{\varphi_0}u_0 (x, \xi),
\end{cases}
\end{align*}
where $h$ and $\Psi$ are defined as in the statement above.
Using the method of characteristics, we obtain our desired results.
\end{proof}
\subsection{Characterization of wave front set}
We remark on the following characterization of the wave front set to prove our main results.
\begin{prop}[K. Kato, M. Kobayashi and S. Ito~\cite{KKI3}]\label{follandokaji}
Let $\left(x_{0}, \xi_{0}\right)\in\mathbb{R}^{n} \times\left(\mathbb{R}^{n} \backslash\{0\}\right)$, $f \in \mathcal{S}'\left(\mathbb{R}^{n}\right)$ and $0<b<1$. The following conditions are equivalent.
\renewcommand{\theenumi}{\roman{enumi}}
\begin{enumerate}
\item $\left(x_{0}, \xi_{0}\right) \notin WF(f)$
\item There exist a neighborhood $K$ of $x_{0}$ and a conic neighborhood $\Gamma$ of $\xi_{0}$ such that for all $N \in \mathbb{N}, a \geq 1$ and $\varphi(x) \in \mathcal{S}\left(\mathbb{R}^{n}\right) \backslash\{0\}$ there exists a constant $C_{N, a, \varphi}>0$ satisfying
\begin{align}\label{FO}
\left|W_{\varphi_{\lambda}} f(x, \lambda \xi)\right| \leq C_{N, a, \varphi} \lambda^{-N}
\end{align}
for all $\lambda \geq 1$, $x \in K$, and $\xi \in \Gamma_a$, where $\varphi_{\lambda}(x)=\lambda^{\frac{n b}{2}} \varphi\left(\lambda^{b} x\right)$.
\item There exist $\varphi(x) \in \mathcal{S}\left(\mathbb{R}^{n}\right) \backslash\{0\}$, a neighborhood $K$ of $x_{0}$ and a conic neighborhood $\Gamma$ of $\xi_{0}$ such that for all $N \in \mathbb{N}$ and $a \geq 1$ there exists a constant $C_{N, a}>0$ satisfying the inequality \eqref{FO} for all $\lambda \geq 1, x \in K$ and $\xi \in \Gamma_a$, where $\varphi_{\lambda}(x)=\lambda^{\frac{nb}{2}} \varphi\left(\lambda^{b} x\right)$.
\end{enumerate}
\end{prop}
\begin{rem}
This type theorem has been originally obtained by G. B. Folland \cite{F} and refined by T. \=Okaji \cite{O1} and \cite{O2}.
\end{rem}
\par
\section{Key lemmas}\label{sec3}
In this section, we prepare some
lemmas for the proof of Theorem \ref{mt1}.
\par
Now let us prepare following key lemmas to show Theorem \ref{mt1}.
In what follows, we write $x^{\ast}_s=x\left(s ; t_{0}, x, \lambda \xi\right)$, $\xi^{\ast}_s=\xi\left(s ; t_{0}, x, \lambda \xi\right)$, $s^{\ast}=s-t_{0}$ and $\varphi_{\lambda}(x)=\left(\varphi_{0}\right)_{\lambda}(x)$ for brevity.
\par
\begin{lem}\label{aj}
Take $b=\min\left(\frac{1}{8},\frac{1-\rho}{8}\right)$.
Let $\bm{a}(t,x)$ satisfy Assumption \ref{A2}, $u(t,x)\in C(\R;L^2(\R^n))$, a multi-indices $\alpha\in\Z^n_+$ be $|\alpha|\geq 2$.
Assume that for $\sigma\geq0$, and $a\geq1$, there exists $C_{\sigma, a, \varphi_0}>0$ satisfying
\begin{align}
\left|W_{\varphi_{\lambda}^{\left(s^{\ast}\right)}} u\left(t, x^{\ast}_s, \xi^{\ast}_s\right)\right| \leq C_{\sigma, a, \varphi_{0}} \lambda^{-\sigma}
\end{align}
for all $\varphi_0(x)\in\Sh(\R^n)\setminus\{0\}$, $x\in K$, $\xi\in\Gamma_a$, $\la\geq1$, and $0\leq s\leq t_0$.
Then there exists $C_{t_0,\alpha,\sigma, a, \varphi_0}>0$ such that 
\begin{align}\label{r}
\left|\int_0^{t_0}\left|R_1\left[\varphi_{\la}^{(s^{\ast})},u\right](s, x^{\ast}_s, \xi^{\ast}_s)\right|ds\right|\leq C_{t_0,\alpha,\sigma, a, \varphi_0}\lambda^{-\sigma-\delta}
\end{align}
for $x\in K$, $\xi\in\Gamma_a$, $\la\geq1$, $0\leq s\leq t_0$ and $\delta\in(2-2b-\rho,2b)$.
\end{lem}
\begin{proof}
We only estimate the following term
\begin{align*}
&\left|\int_0^{t_0}\left|\int \overline{\varphi_{\la}^{(s^{\ast})}(y-x^{\ast}_s)}u(s,y)e^{-iy\cdot\xi^{\ast}_s} {\xi^{\ast}_{s,j}}r_{1,j}^L(s,x^{\ast}_s,y)dy\right|ds\right|
\end{align*}
for $0\leq s\leq t_0$ and $j=1,\ldots,n$.
For the others, it can be obtained by the same way discussed in \cite{KI}.
\par
By virtue of Lemma \ref{L2}, there exists $C_{\delta}>0$ such that
\begin{align*}
\int_0^{t_0}\frac{\left|\xi^{\ast}_s\right|}{\langle x^{\ast}_s\rangle^{1+\delta}}ds\leq C_{\delta}(1+t_0)
\end{align*}
for $x\in K$ and $\xi\in\Gamma_a$.
So it suffices to show that there exists $C_{\alpha,\sigma, a, \varphi_0}>0$ such that 
\begin{align}\label{r2}
\left|\int_0^{t_0}\left|\int \overline{\varphi_{\la}^{(s^{\ast})}(y-x^{\ast}_s)}u(s,y)e^{-iy\cdot\xi^{\ast}_s} {\xi^{\ast}_{s,j}}r_{0,0}^{L,j}(s,x^{\ast}_s,y)dy\right|ds\right|\leq C_{\alpha,\sigma, a, \varphi_0}\lambda^{-\sigma-2b}
\end{align}
for $x\in K$, $\xi\in\Gamma_a$, $\la\geq1$, $0\leq s\leq t_0$ and $\delta\in(2-2b-\rho,2b)$.
\par
Taking $\ep>0$ as $\ep=(1-\rho)/2$, the assumption and the equality \eqref{a3} yield that
\begin{align*}
&\left|\int \overline{\varphi_{\la}^{(s^{\ast})}(y-x^{\ast}_s)}u(s,y)e^{-iy\cdot\xi^{\ast}_s} {\xi^{\ast}_{s,j}}r_{0,0}^{L,j}(s,x^{\ast}_s,y)dy\right|\\
&\leq \sum_{2 \leq|\alpha| \leq L-1} \sum_{\substack{\beta+\gamma=\alpha \\ \beta^{\prime} \leq \beta, \gamma^{\prime} \leq \gamma}}\frac{C_{\beta, \gamma, \beta^{\prime}, \gamma^{\prime}}}{\alpha !}\left|\partial_{x}^{\alpha} a_j\left(s, x^{\ast}_s\right)\right|
\left|\xi^{\ast}_{s,j}\right|
s^{\ast|\beta|} \lambda^{b(|\beta|-|\gamma|)}\\
&\qquad\times\left|W_{(\varphi_0^{\beta^{\prime}, \gamma^{\prime}})^{(s^{\ast})}_{\lambda}} u\left(s, x^{\ast}_s, \xi^{\ast}_s\right)\right| \\
&\leq \sum_{2 \leq|\alpha| \leq L-1} \sum_{\beta+\gamma=\alpha} \frac{C_{\beta, \gamma}}{\alpha !}\frac{\left|\xi^{\ast}_{s}\right|}{\langle x^{\ast}_s\rangle^{1+\ep}}  \langle x^{\ast}_s\rangle^{1+\ep+\rho-|\alpha|} s^{\ast|\beta|}\lambda^{b(|\beta|-|\gamma|)} \lambda^{-\sigma}
\end{align*}
for $s\in[0,t_0]$, where $(\varphi_0^{\beta^{\prime}, \gamma^{\prime}})^{(s^{\ast})}_{\lambda}=e^{is^{\ast}\Delta/2}(x^{\gamma'}\partial_{x}^{\beta'}\varphi_0)_{\lambda}$.
In order to show \eqref{r2}, it only remains to show that
\begin{align}\label{sup}
\sup_{s\in[0,t_0]}\left(\langle x^{\ast}_s\rangle^{1+\ep+\rho-|\alpha|} {s^{\ast}}^{|\beta|}\right) \lambda^{b(|\beta|-|\gamma|)} \leq C_{t_0,\alpha}\la^{-2b}
\end{align}
for some  $C_{t_0, \alpha}>0$, any $\alpha, \beta, \gamma\in\Z^n_{+}$ with $2\leq|\alpha|$ and $\beta+\gamma=\alpha$.
Indeed, if the above inequality holds, then it holds that
\begin{align*}
&\left|\int_0^{t_0}\left|\int \overline{\varphi_{\la}^{(s^{\ast})}(y-x^{\ast}_s)}u(s,y)e^{-iy\cdot\xi^{\ast}_s} {\xi^{\ast}_{s,j}}r_{0,0}^{L,j}(s,x^{\ast}_s,y)dy\right|ds\right|\\
&\leq\sum_{2 \leq|\alpha| \leq L-1} \sum_{\beta+\gamma=\alpha} \frac{C_{\beta, \gamma}}{\alpha !}\int_0^{t_0}\frac{\left|\xi^{\ast}_{s}\right|}{\langle x^{\ast}_s\rangle^{1+\ep}}\langle x^{\ast}_s\rangle^{1+\ep+\rho-|\alpha|} s^{\ast|\beta|}ds\lambda^{b(|\beta|-|\gamma|)} \lambda^{-\sigma}\\
&\leq\sum_{2 \leq|\alpha| \leq L-1} \sum_{\beta+\gamma=\alpha} \frac{C_{\beta, \gamma}}{\alpha !}
\int_0^{t_0}\frac{\left|\xi^{\ast}_{s}\right|}{\langle x^{\ast}_s\rangle^{1+\ep}}ds
\sup_{s\in[0,t_0]}\left(\langle x^{\ast}_s\rangle^{1+\ep+\rho-|\alpha|} s^{\ast|\beta|}\right) \lambda^{b(|\beta|-|\gamma|)} \lambda^{-\sigma}\\
&\leq\sum_{2 \leq|\alpha| \leq L-1} \sum_{\beta+\gamma=\alpha} \frac{C_{\beta, \gamma}}{\alpha !}
C_{t_0,\alpha}C_{\delta}(1+t_0)\lambda^{-\sigma-2b}
\leq C\lambda^{-\sigma-2b},
\end{align*}
which implies \eqref{r2}.
We estimate the left hand side of \eqref{sup} by separating into following two cases; (I) $0\leq|s^{\ast}|\leq \la^{-2b}$ and (II) $\la^{-2b}\leq|s^{\ast}|\leq t_0$.
\par
(I) For $s$ satisfying $0\leq|s^{\ast}|\leq \la^{-2b}$, since $1+\ep+\rho-2\leq0,$ we have
\begin{align*}
\langle x^{\ast}_s\rangle^{1+\ep+\rho-|\alpha|} s^{\ast|\beta|}\lambda^{b(|\beta|-|\gamma|)}
&\leq \lambda^{-b(|\beta|+|\gamma|)}\leq\la^{-2b}
\end{align*}
\par
(II) For $s$ satisfying $\la^{-2b}\leq|s^{\ast}|\leq t_0,$ since $b=\min\left(\frac{1}{8},\frac{1-\rho}{8}\right)$, \eqref{a1} yields
\begin{align*}
\langle x^{\ast}_s\rangle^{1+\ep+\rho-|\alpha|} \lambda^{b(|\beta|-|\gamma|)}s^{\ast|\beta|}
&\leq C_{a,s^{\ast}}\left(\la|s^{\ast}|\right)^{1+\ep+\rho-|\alpha|}\left(\lambda^{b}|s^{\ast}|\right)^{|\alpha|}\\
&\leq  C_{a,s^{\ast}}t_0^{1+\ep+\rho}\la^{1+\ep+\rho-(1-b)|\alpha|}\\
&\leq  C_{a,s^{\ast}}t_0^{1+\ep+\rho}\la^{2b+1+\ep+\rho-2}\leq t_0^{1+\ep+\rho}\la^{-2b}
\end{align*}
for all $\lambda\geq 1$ and $\beta, \gamma\in\Z^n_{+}$ with $\beta+\gamma=\alpha$,
where $C_{a,s^{\ast}}$ only depends on $a$ and $s^{\ast}$.
Therefore, we obtain \eqref{sup}. 
\end{proof}
\begin{lem}\label{R2}
Let $\bm{a}(t,x)$ satisfy Assumption \ref{A2}, $u(t,x)\in C(\R;L^2(\R^n))$ and a multi-indices $\alpha\in\Z^n_+$ be $|\alpha|= L$ with $L\in\N\cup\{0\}$.
Then for all $a\geq1$ and $k\in\N$ there exists $C_{\alpha, a, \varphi_0}>0$ such that 
\begin{align}\label{r}
\left|\int_0^{t_0}\left|R_2\left[\varphi_{\la}^{(s^{\ast})},u\right](s, x_s^{\ast}, \xi_s^{\ast})\right|ds\right|\leq C_{\alpha,\sigma, a, \varphi_0}\lambda^{-k}
\end{align}
for $x\in K$, $\xi\in\Gamma_a$, $\la\geq1$ and $0\leq s\leq t_0$ if we take $L$ sufficiently large.
\end{lem}
\begin{rem}
In order to show the decay of $\lambda$ in \eqref{r}, we do not need
to assume $u(t)$ has a decay of $\lambda$.
By virtue of Lemma~\ref{L3}, we can obtain the decay of $\lambda$ from $(y-x_s^{\ast})^{\alpha}\cdot\varphi^{(s^{\ast})}_{\lambda}(y-x_s^{\ast})$ if we have polynomials $(y-x_s^{\ast})^{\alpha}$ with its order $|\alpha|$ sufficiently large as demonstrated below.
Polynomials $y-x_s^{\ast}$ of desired degree are obtained from the Taylor expansion of the potential $\bm{a}(t,x)$ since it is smooth.
\end{rem}
\begin{proof}
\par
We take $\psi_{1}$ and $\psi_{2}$ belonging to $C^{\infty}(\R)$ as
$$
\begin{array}{l}
\psi_{1}(s)=\left\{\begin{array}{ll}
1 & \text { for } s \leq 1, \\
0 & \text { for } s \geq 2,
\end{array}\right. \\
\psi_{2}(s)=1-\psi_{1}(s), \quad  s \in \mathbb{R}.
\end{array}
$$
Take $d$ with $0<d<b$ and
define $r_{1,0}^{\alpha,j}(t,x,y)=\int_0^1\partial_{x}^{\alpha}a_j\left(t,y(x,\theta)\right)(1-\theta)^{L-1}d\theta$.
We devide one of terms in $R_2\left[\varphi_{\lambda}^{(s)},u\right](s, x(s), \xi(s))$
into two parts as
\begin{align*}
\sum_{j=1}^n\int u(s,y)e^{-iy\cdot\xi^{\ast}_s} \xi^{\ast}_{j,s}r^{L,j}_{1,0}(s,x^{\ast}_s,y)
\overline{\varphi_{\lambda}^{(s)}(y-x^{\ast}_s)}dy=\sum_{l=1}^2I_{\alpha, l}\left(s, x^{\ast}_s, \xi^{\ast}_s, \lambda\right),
\end{align*}
where
\begin{align*}
I_{\alpha, l}\left(s, x^{\ast}_s, \xi^{\ast}_s, \lambda\right)&=\sum_{j=1}^n\xi_{j,s}^{\ast}\iiint \psi_{l}\left(\frac{\lambda^{d}\left|y-x^{\ast}_s\right|}{1+\lambda\left|s^{\ast}\right|}\right) r_{1,0}^{\alpha,j}\left(s, x^{\ast}_s, y\right)\left(y-x^{\ast}_s\right)^{\alpha} \\
& \times \varphi_{\lambda}^{\left(s^{\ast}\right)}\left(y-x^{\ast}_s\right) \varphi_{\lambda}^{\left(s^{\ast}\right)}(y-z) W_{\varphi_{\lambda}^{\left(s^{\ast}\right)}} u(s, z, \eta) e^{-i y\left(\xi^{\ast}_s-\eta\right)} dzd\eta dy
\end{align*}
for $l=1,2$.
Since the other terms in $R_2\left[\varphi_{\lambda}^{(s)},u\right](s, x(s), \xi(s))$ can be estimated in the same way, we only show that
for $l=1,2$, there exists a positive constant $C_{\sigma, a, \varphi_{0}}$ such that
\begin{align}\label{il}
\int_0^{t_0}\left|I_{\alpha, l}\left(s, x^{\ast}_s, \xi^{\ast}_s, \lambda\right)\right|ds \leq C_{\sigma, a, \varphi_{0}} \lambda^{-\sigma-2 b}
\end{align}
for $\lambda \geq 1, x \in K, \xi \in \Gamma_a$ and $0 \leq s \leq t_{0}$.
For the convinience of the proof, we employ integration by parts and the fact that $\left(1-\Delta_{y}\right) e^{i y(\xi-\eta)}=\left(1+|\xi-\eta|^{2}\right) e^{i y(\xi-\eta)}$ to $I_{\alpha, l}$ then
\begin{align*}
I_{\alpha, l}\left(s, x^{\ast}_s, \xi^{\ast}_s, \lambda\right)&=\sum_{j=1}^n\xi_{s,j}^{\ast}\iiint\left(1+|\xi^{\ast}_s-\eta|^{2}\right)^{-N}\\
&\times\left(1-\Delta_{y}\right)^{N}\left[\overline{\varphi_{\lambda}^{\left(s^{\ast}\right)}\left(y-x^{\ast}_s\right)} \varphi_{\lambda}^{\left(s^{\ast}\right)}(y-z) \psi_{l}\left(\frac{\lambda^{d}\left|y-x^{\ast}_s\right|}{1+\lambda\left|s^{\ast}\right|}\right)\right.\\
&\left.\times r_{1,0}^{\alpha,j}\left(s, x^{\ast}_s, y\right)\left(y-x^{\ast}_s\right)^{\alpha}\right] W_{\varphi_{\lambda}^{\left(s^{\ast}\right)}} u(s, z, \eta) e^{-i y\left(\xi^{\ast}_s-\eta\right)} dzdy d \eta\\
&=\sum_{|\alpha_1|+\cdots+|\alpha_5|\leq 2N}C_{\alpha_1,\ldots,\alpha_5}\sum_{j=1}^n\xi_{s,j}^{\ast}\iiint\left(1+|\xi^{\ast}_s-\eta|^{2}\right)^{-N}\\
&\times\overline{\partial_y^{\alpha_1}\varphi_{\lambda}^{\left(s^{\ast}\right)}\left(y-x^{\ast}_s\right)} \partial_y^{\alpha_2}\varphi_{\lambda}^{\left(s^{\ast}\right)}(y-z) \partial_y^{\alpha_3}\psi_{l}\left(\frac{\lambda^{d}\left|y-x^{\ast}_s\right|}{1+\lambda\left|s^{\ast}\right|}\right)\\
&\times \partial_y^{\alpha_4}r_{1,0}^{\alpha,j}\left(s, x^{\ast}_s, y\right)\partial_y^{\alpha_5}\left(y-x^{\ast}_s\right)^{\alpha} W_{\varphi_{\lambda}^{\left(s^{\ast}\right)}} u(s, z, \eta) e^{-i y\left(\xi^{\ast}_s-\eta\right)} dzdy d \eta .
\end{align*}

We take $d^{\prime}$ such that $0<d^{\prime}<d$.
Since $\left|y-x^{\ast}_s\right| \leq 2\left(1+\lambda\left|s^{\ast}\right|\right) \lambda^{-d}$ if $\psi_{1}\left(\frac{\lambda^{d}\left|y-x^{\ast}_s\right|}{1+\lambda\left|s^{\ast}\right|}\right) \neq 0$, Lemma~\ref{L1} yields
\begin{align}\label{ineq:3.8}
&\left|\partial_{y}^{\alpha_4}r_{1,m}^{\alpha,j}\left(s, x^{\ast}_s, y\right)\partial_{y}^{\alpha_5}\left(y-x^{\ast}_s\right)^{\alpha}\right|\\
& \leq C\left(1+\left|x^{\ast}_s+\theta\left(y-x^{\ast}_s\right)\right|\right)^{\rho-L}\left(1+\lambda\left|s^{\ast}\right|\right)^{L} \lambda^{-d L} \notag\\
& \leq C\frac{\left(1+\left|x^{\ast}_s\right|-\left|y-x^{\ast}_s\right|\right)^{\rho+1+\ep-L}}{\left(1+\left|x^{\ast}_s\right|\right)^{1+\ep}}\left(1+\lambda\left|s^{\ast}\right|\right)^{L+1+\ep} \lambda^{-d(L+1+\ep)}\notag\\
& \leq C\left(1+\lambda\left|s^{\ast}\right|\right)^{\rho+2(1+\ep)} \lambda^{-d(L+1+\ep)}\left(1+\left|x^{\ast}_s\right|\right)^{-(1+\ep)}\notag
\end{align}
for $m=0,1$, $\lambda^{d'-1}\leq|s^{\ast}|\leq t_0$ and $\lambda\geq\lambda_1$ with some $\lambda_1\geq1$ sufficiently large.
Simple calculation and \eqref{a3} yield that 
\begin{align}
&\left\|\partial_{y}^{\beta} \varphi_{\lambda}^{\left(s^{\ast}\right)}\right\|_{L^{p}} \leq C \lambda^{b(|\beta|+(2-p)(2(n+1)-n/2))},\\ 
&\left|\partial_{y}^{\beta}\left\{\psi_{l}\left(\frac{\lambda^{d}\left|y-x^{\ast}_s\right|}{1+\lambda\left|s^{\ast}\right|}\right)\right\}\right| \leq C \lambda^{d|\beta|}\label{ineq:3.10}
\end{align}
for $y\in\R^n$, $l,p=1,2$ and $\beta\in\Z^n_{+}$.
From above inequalities \eqref{ineq:3.8}-\eqref{ineq:3.10}, Young's inequality and \eqref{a2}, the L.H.S of \eqref{il} with $l=1$ is bounded by
\begin{align*}
&C\left(1+\lambda\left|s^{\ast}\right|\right)^{\rho+2(1+\ep)} \lambda^{-d(L+1+\ep-2N)}\\
&\times\sum_{|\alpha_1+\alpha_2|\leq2N}\int_0^{t_0}\frac{\left|\xi_{s}^{\ast}\right|}{(1+|x_s^{\ast}|)^{1+\varepsilon}}ds\sup_{s\in[0,t_0]}\iiint\left(1+|\xi^{\ast}_s-\eta|^{2}\right)^{-N}\\
&\times \left|{\partial_y^{\alpha_1}\varphi_{\lambda}^{\left(s^{\ast}\right)}\left(y-x^{\ast}_s\right)}\right|\left|\partial_y^{\alpha_2}\varphi_{\lambda}^{\left(s^{\ast}\right)}(y-z)\right|
\left|W_{\varphi_{\lambda}^{\left(s^{\ast}\right)}} u(s, z, \eta)\right| d y d \eta dz\\
&\leq C_{t_0}\left(1+\lambda\left|s^{\ast}\right|\right)^{\rho+2(1+\ep)} \lambda^{-d(L+1+\ep-2N)} \\
&\times\sum_{|\alpha_1+\alpha_2|\leq 2N}\lambda^{b|\alpha_1|}\sup_{s\in[0,t_0]}\left\|\left[\left|\partial_y^{\alpha_2}\varphi_{\lambda}^{\left(s^{\ast}\right)}\right| \ast\left|W_{\varphi_{\lambda}^{\left(s^{\ast}\right)}} u(s, \cdot, \eta)\right|\right](y)\right\|_{L^2(\R^n_{y}\times\R^n_{\eta})}
\\
&\leq C_{t_0}\left(1+\lambda\left|s^{\ast}\right|\right)^{\rho+2(1+\ep)} \lambda^{-d(L+1+\ep-2N)} \lambda^{b(4(n+1)-n/2)}
\sup_{s\in[0,t_0]}\|u(s)\|_{L^2}
\end{align*}
for $\left|s^{\ast}\right| \geq \lambda^{d^{\prime}-1}$ and $\lambda \geq \lambda_1$, where $N= n+1$.

On the other hand, when $\left|s^{\ast}\right| \leq \lambda^{d^{\prime}-1}$, we have $\left|y-x^{\ast}_s\right| \leq C\left(1+\lambda\left|s^{\ast}\right|\right) \lambda^{-d} \leq C \lambda^{d^{\prime}-d}$ and this shows that
\begin{align*}
\int_0^{t_0}\left|I_{\alpha, 1}\left(s, x^{\ast}_s, \xi^{\ast}_s, \lambda\right)\right|ds&\leq C_{t_0}\lambda^{2N-\rho+b(2N+(2(n+1)-n/2))}\lambda^{-(d-d')L}.
\end{align*}
Hence combining those calculations, \eqref{il} with $l=1$ holds if we take $L$ sufficiently large.
\par
Finally, we estimate $I_{\alpha, 2}$.
By multiplying $\left(1+\left|y-x^{\ast}_s\right|^{2}\right)^{-M}\left(1+\left|y-x^{\ast}_s\right|^{2}\right)^{M}$ to the integrand of $I_{\alpha,2}$,
the inequality \eqref{a4} in Lemma \ref{L3} yields
\begin{align*}
&|\xi_{s,j}^{\ast}|\iiint\left(1+|\xi^{\ast}_s-\eta|^{2}\right)^{-N}\left(1+\left|y-x^{\ast}_s\right|^{2}\right)^{-M}\left(1+\left|y-x^{\ast}_s\right|^{2}\right)^{M}\\
&\times|\partial_y^{\alpha_1}\varphi_{\lambda}^{\left(s^{\ast}\right)}\left(y-x^{\ast}_s\right)| |\partial_y^{\alpha_2}\varphi_{\lambda}^{\left(s^{\ast}\right)}(y-z)|
\left|\partial_y^{\alpha_3}\psi_{2}\left(\frac{\lambda^{d}\left|y-x^{\ast}_s\right|}{1+\lambda\left|s^{\ast}\right|}\right)\right|\\
&\times |\partial_y^{\alpha_4}r_{1,0}^{\alpha,j}\left(s, x^{\ast}_s, y\right)||\partial_y^{\alpha_5}\left(y-x^{\ast}_s\right)^{\alpha}||W_{\varphi_{\lambda}^{\left(s^{\ast}\right)}} u(s, z, \eta)|dzdy d \eta \\
&\leq \lambda^{d|\alpha_3|}\frac{|\xi_{s}^{\ast}|}{\langle x_s^{\ast}\rangle^{L-\rho}}\iiint\left(1+|\xi^{\ast}_s-\eta|^{2}\right)^{-N}\left(1+\left|y-x^{\ast}_s\right|^{2}\right)^{-M}\left(1+\left|y-x^{\ast}_s\right|^{2}\right)^{M+L-\rho}\\
&\times|y-x^{\ast}_s|^{L-|\alpha_5|}|\partial_y^{\alpha_1}\varphi_{\lambda}^{\left(s^{\ast}\right)}\left(y-x^{\ast}_s\right)| |\partial_y^{\alpha_2}\varphi_{\lambda}^{\left(s^{\ast}\right)}(y-z)|
|W_{\varphi_{\lambda}^{\left(s^{\ast}\right)}} u(s, z, \eta)|dzdy d \eta \\
&\leq C\sum_{\substack{L\leq|\beta+\gamma|\leq M+2L}}|s^{\ast}|^{|\gamma|}\lambda^{d|\alpha_3|+b(|\alpha_1|+|\alpha_2|+|\alpha_5|+|\gamma|-|\beta|)}\frac{|\xi_{s}^{\ast}|}{\langle x_s^{\ast}\rangle^{L-\rho}}\\
&\times\iiint\left(1+|\xi^{\ast}_s-\eta|^{2}\right)^{-N}\left(1+\left|y-x^{\ast}_s\right|^{2}\right)^{-M}\left|\left(\varphi^{\alpha_1+\gamma,\beta-\alpha_5}\right)_{\lambda}^{\left(s^{\ast}\right)}\left(y-x^{\ast}_s\right)\right| \\
&\times\left|\left(\varphi^{\alpha_2,0}\right)_{\lambda}^{\left(s^{\ast}\right)}(y-z)\right|
|W_{\varphi_{\lambda}^{\left(s^{\ast}\right)}} u(s, z, \eta)|dzdy d \eta.
\end{align*}
Since $\left|y-x^{\ast}_s\right| \geq \lambda^{-d}\left(1+\lambda\left|s^{\ast}\right|\right)$ if $\displaystyle\psi_{2}\left(\frac{\lambda^{d}\left|y-x^{\ast}_s\right|}{1+\left|s^{\ast}\right| \lambda}\right) \neq 0$, we have with $M=m+n+1$ and $N=n+1$
\begin{align*}
\left|I_{\alpha, 2}\right|
&\leq \sum_{\substack{\left|\alpha_{1}+\cdots+\alpha_{5}\right| \leq 2 N\\ L\leq|\beta+\gamma|\leq M+2L}}|s^{\ast}|^{|\gamma|}\lambda^{d|\alpha_3|+b(|\alpha_1|+|\alpha_2|+|\alpha_5|+|\gamma|-|\beta|)}\frac{|\xi_{s}^{\ast}|}{\langle x_s^{\ast}\rangle^{L-\rho}}\\
&\times\left(1+\lambda^{-2 d}\left(1+\lambda\left|s^{\ast}\right|\right)^{2}\right)^{-m} \left\|\left(1+|y|^{2}\right)^{-n-1}\right\|_{L_{y}^{1}}\left\|\left(1+|\eta|^{2}\right)^{-n-1}\right\|_{L_{\eta}^{1}}\\
&\times\left\|\varphi_0^{\alpha_1+\gamma,\beta-\alpha_5}\right\|_{L_{y}^{2}}\left\|\varphi_0^{\alpha_2,0}\right\|_{L_{z}^{2}}\left\|W_{\varphi_{\lambda}^{\left(s^{\ast}\right)}} u(s, z, \eta)\right\|_{L_{z, \eta}^{2}}.
\end{align*}
For $0 \leq |s^{\ast}| \leq \lambda^{-2 b}$, we have $\left|s^{\ast}\right| \lambda^{b} \leq \lambda^{-b}$. Hence we obtain
$$
\begin{aligned}
\int_0^{t_0}\left|I_{\alpha, 2}\right|ds
& \leq C_{t_0} \lambda^{-b(L-2 N)}=C_{t_0} \lambda^{-b(L-2(n+1))} \leq C_{t_0} \lambda^{-2 b-\sigma},
\end{aligned}
$$
if we take $L\in\N$ being grater than $N+2 n+4+\sigma / b$.
For $\lambda^{-2 b} \leq |s^{\ast}| \leq t_{0}$, we have
\begin{align*}
&\sum_{\substack{\left|\alpha_{1}+\cdots+\alpha_{5}\right| \leq 2 N\\ L \leq|\beta+\gamma| \leq M+2L}} C\left(1+\lambda^{-2 d}\left(1+\lambda\left|s^{\ast}\right|\right)^{2}\right)^{-m}\left(\lambda^{b}\left|s^{\ast}\right|\right)^{|\gamma|} \lambda^{b\left(2N-|\beta|\right)} \lambda^{d\left|\alpha_{3}\right|} \\
&\leq C\left(1+\left(\lambda^{1-d-2 b}\right)^{2}\right)^{-m} \lambda^{b(2 M+2 N+L)}\\
&\leq C \lambda^{-2 m(1-d-2 b)} \lambda^{b(2 m+4(n+1)+L)} \\
&\leq C \lambda^{-2 m(1-d-3 b)} \lambda^{b(4(n+1)+L)}.
\end{align*}
Since $1-d-3b>1-4 b \geq 0$, we have $\left|I_{\alpha, 2}\right| \leq C \lambda^{-2 b-\sigma}$, if we take $m$ sufficiently large. This shows \eqref{il} with $l=2$ for $x \in K, \xi \in \Gamma_a$, $\lambda \geq 1$ and $0 \leq s \leq t_{0}$.
\end{proof}
\par
\section{Proof of Theorem \ref{mt1}}\label{sec4}
In this section, we will prove Theorem \ref{mt1}.
\par
\begin{proof}[Proof of Theorem \ref{mt1}]
It is trivial that (ii) implies (iii).
We only show that (iii) implies (i) since we can show that (i) implies (ii) in the similar way.
Let $K$, $\Gamma\subset\R^n$ and $\varphi_0\in\CAL{S}(\R^n)$ be a neighborhood of $x_0$, a conic neighborhood of $\xi_0$ and the basic wave packet satisfying \eqref{main} stated in (iii), respectively.
It suffices to show that the following assertion $P\left(\sigma\right)$ holds for all $\sigma \geq 0$ under the condition of (iii).
\par
$P\left(\sigma\right): ``$ There exists a positive constant $C_{\sigma, a, \varphi_{0}}$ such that
\begin{align}\label{assertion}
\left|W_{\varphi_{\lambda}^{\left(t-t_{0}\right)}} u\left(t, x\left(t ; t_{0}, x, \lambda \xi\right), \xi\left(t ; t_{0}, x, \lambda \xi\right)\right)\right| \leq C_{\sigma, a, \varphi_{0}} \lambda^{-\sigma}
\end{align}
for all $x \in K$, $\xi \in \Gamma_a$, $\lambda \geq 1$ and $0 \leq t \leq t_{0}$."
\par
In fact, taking $t=t_{0}$, we have $\varphi_{\lambda}^{\left(t_{0}-t_{0}\right)}=\varphi_{0,\lambda}$, 
$x\left(t_{0} ; t_{0}, x, \lambda \xi\right)=x$ and\\
$\xi\left(t_{0} ; t_{0}, x, \lambda \xi\right)=\lambda \xi$. Therefore, from \eqref{assertion}, we have immediately
$$
\left|W_{\varphi_{0,\lambda}} u\left(t_{0}, x, \lambda \xi\right)\right| \leq C_{\sigma, a, \varphi_{0}} \lambda^{-\sigma}
$$
for $\lambda \geq 1$, $x \in K$ and $\xi \in \Gamma_a$.
This and Proposition \ref{follandokaji} imply (i).
We show by inductive method with respect to $\sigma$ that $P\left(\sigma\right)$ holds for all $\sigma \geq 0$.
\par
First we show that $P\left(0\right)$ holds. Since $u_{0}(x) \in L^{2}\left(\mathbb{R}^{n}\right)$ and then $u(t, x) \in C\left(\mathbb{R} ; L^{2}\left(\mathbb{R}^{n}\right)\right)$, Schwarz's inequality and the conservation of $L^{2}$ norm of solutions of \eqref{eq1} show that
\begin{align*}
&\left|W_{\varphi_{\lambda}^{\left(t-t_{0}\right)}} u\left(t, x\left(t ; t_{0}, x, \lambda \xi\right), \xi\left(t ; t_{0}, x, \lambda \xi\right)\right)\right| \\
&\leq \int\left|\varphi_{\lambda}^{\left(t-t_{0}\right)}\left(y-x\left(t ; t_{0}, x, \lambda \xi\right)\right)\right|\left| u(t, y)\right| d y \\
&\leq\left\|\varphi_{\lambda}^{\left(t-t_{0}\right)}(\cdot)\right\|_{L^{2}}\|u(t, \cdot)\|_{L^{2}} \\
&=\left\|\varphi_{\lambda}(\cdot)\right\|_{L^{2}}\left\|u_{0}(\cdot)\right\|_{L^{2}}=\left\|\varphi_{0}(\cdot)\right\|_{L^{2}}\left\|u_{0}(\cdot)\right\|_{L^{2}},
\end{align*}
which shows $P\left(0\right)$ is valid.
\par
Next we show that for a fixed $P\left(\sigma+2 b\right)$ holds under the assumption that $P\left(\sigma\right)$ holds. To do so, it suffices to show that there exists a positive constant $C_{a, \varphi_{0}}$ such that
\begin{align}\label{remainder}
\left|R\left[\varphi_{\la}^{(s)},u\right]\left(s, x^{\ast}_s, \xi^{\ast}_s\right)\right| \leq C_{a, \varphi_{0}} \lambda^{-(\sigma+2 b)}
\end{align}
for all $x \in K$, $\xi \in \Gamma_a$, $\lambda \geq 1$ and $0 \leq s \leq t_{0}$, since the first term of the right hand side of \eqref{fo} is estimated by $C \lambda^{-(\sigma+2 b)}$ from the assumption on $u_{0}$.

Indeed, \eqref{remainder} can be shown by combining Lemma \ref{aj} and \ref{R2} if we take a sufficiently large $L\in\N$.
\end{proof}
\par
\section{Proof of Corollaries}
\begin{proof}[Proof of Corollary \ref{cor1}]
First, we shortly note that if the initial value problem \eqref{eq1} is well-posed (locally in time) on $\CAL{H}^k$ for $k\in\Z$, where $\CAL{H}^k=\{f\in\CAL{S}'(\R^n)\mid \sum_{|\alpha+\beta|\leq k}$\\
$x^{\alpha}\partial_{x}^{\beta}f\in L^2(\R^n)\}$
for $k\geq0$ and $\CAL{H}^{-k}$ is defined as its dual space, we can apply Theorem~\ref{mt1} to the solution $u(t)$ of \eqref{eq1}, which belongs to $C(\R;\CAL{H}^{k})$, with small modifications: we just replace all $\varphi^{(t)}_{\lambda}$'s appearing in the statement and the proof of the theorem with $\varphi_{\lambda,k}^{(t)}=\lambda^{kb}\varphi_{\lambda}^{(t)}$ when $k<0$.
Consequently we can analyze $WF(E(t_0,\cdot,0))$ by Theorem~\ref{mt1} since \eqref{eq1} is indeed well-posed on $\CAL{H}^k$ for any $k\in\Z$ under our assumption on $\bm{a}(t,x)$ and there exists some $k_0\in\Z_{\leq0}$ such that $\delta_0\in H^{k_0}(\R^n)\subset \CAL{H}^{k_0}$. 

We can assume $t_0>0$ without any loss of generality.
Let $(x_0, \xi_0)\in \R^n\times \R^n\setminus\{0\}$ and $\varphi_0(x)=e^{-|x|^2/2}.$
From Theorem \ref{mt1}, it suffices to show that there exist a neighborhood $K$ of $x_0$ and a conic neighborhood $\Gamma$ of $\xi_0$ such that for all $N\in\N$ and $a\geq1$, there exists a constant $C_{N,a}>0$ satisfying
\begin{align*}
\left|W_{\varphi_{\la}^{(-t_0)}}[\delta_0](x(0;t_0,x,\la\xi),\xi(0;t_0,x,\la\xi))\right|\leq C_{N,a}\la^{-N}
\end{align*}
for $\la\geq1$, $x\in K$ and $\xi\in\Gamma_a=\{\xi\in\Gamma|a^{-1}\leq|\xi|\leq a\}$, where $x(0)=x(0;t_0,x,\la\xi),$ $\xi(0)=\xi(0;t_0,x,\la\xi)$ are solutions to \eqref{ce}.
From the definition of wave packet transform, the left hand side of the above can be rewritten as
\begin{align*}
\left|W_{\varphi_{\la}^{(-t_0)}}[\delta_0](x(0;t_0,x,\la\xi),\xi(0;t_0,x,\la\xi))\right|&=|\varphi_{\la}^{(-t_0)}(-x(0))|\\
&=\left|\int e^{it_0\eta^2/2}\widehat{\varphi_{0,\la}}(\eta)e^{-ix(0)\cdot\eta}d\eta\right|.
\end{align*}
Since $\widehat{\varphi_{0,\la}}(\eta)=\la^{-nb/2}e^{-|\eta|^2/(2\la^{2b})}$,
the right hand side in the last equality of the above can be calculated as
\begin{align*}
\left|\int e^{it_0\eta^2/2}\widehat{\varphi_{0,\la}}(\eta)e^{-ix(0)\cdot\eta}d\eta\right|&=\left|e^{-|x(0)|^2/(2\Lambda)}\int e^{-\frac{\Lambda}{2}\left(\eta+i\frac{x(0)}{\Lambda}\right)^2}d\eta\right|\\
&=C_n\lambda^{-nb/2}|\Lambda|^{n/2}e^{-|x(0)|^2/(2|\Lambda|)},
\end{align*}
where $\Lambda=\la^{-2b}-it_0$.

Since the Gaussian is Schwartz function, we only need to show $|x(0)|^2/(2|\Lambda|)$ has a growth of $\lambda$.
Taking $\lambda$ sufficiently large, we  may estimate $|x(0)|$ as
\begin{align*}
|x(0)|&\geq \lambda t_0|\xi|-|x|-\int_{0}^{\lambda^{p-1}}|s^{\ast}|\left|(\nabla_x\bm{a})^{T}(s,x(s^{\ast}))\right|\left(|\xi(s^{\ast})|+|\bm{a}(s,x(s^{\ast}))|\right)ds\\
&-\int_{\lambda^{p-1}}^{t_0}|s^{\ast}|\left|(\nabla_x\bm{a})^{T}(s,x(s^{\ast}))\right|\left(|\xi(s^{\ast})|+|\bm{a}(s,x(s^{\ast}))|\right)ds-\int^{t_0}_0\left|\bm{a}(s,x(s^{\ast}))\right|ds
\end{align*}
for $0<p<1$.
Employing Lemma~\ref{L1} and its proof, we have
\begin{align*}
&\int_{\lambda^{p-1}}^{t_0}|s^{\ast}|\left|(\nabla_x\bm{a})^{T}(s,x(s^{\ast}))\right|\left(|\xi(s^{\ast})|+|\bm{a}(s,x(s^{\ast}))|\right)ds\\
&\leq (2a)^2\lambda^{\rho}\langle t_0\rangle^{1+\rho}+(2a)^{2\rho}\lambda^{2\rho-1}\langle t_0\rangle^{2\rho+1}\\
&\leq\lambda\left(\frac{(2a)^2\langle t_0\rangle^{1+\rho}}{\lambda^{1-\rho}} +\frac{(2a)^{2\rho}\langle t_0\rangle^{2\rho+1}}{\lambda^{2(1-\rho)}}\right)\leq \lambda\frac{t_0}{4a}
\end{align*}
for sufficiently large $\lambda$.
On the other hand, since for sufficiently small $|\tau^{\ast}|$, it holds that
\begin{align*}
|x(\tau^{\ast})-x-\la\tau^{\ast}\xi|+|\xi(\tau^{\ast})-\la\xi|\leq C|\tau^{\ast}|(1+|x|+\la|\xi|),
\end{align*}
which is a byproduct obtained in the process of the Picard successive iteration,
we have
\begin{align*}
&\int_{0}^{\lambda^{p-1}}|s^{\ast}|\left|(\nabla_x\bm{a})^{T}(s,x(s^{\ast}))\right|\left(|\xi(s^{\ast})|+|\bm{a}(s,x(s^{\ast}))|\right)ds\\
&\leq C\lambda^{p-1}\int_{0}^{\lambda^{p-1}}\langle s^{\ast}\rangle\frac{1+\lambda|\xi|+|x|}{\langle x(s^{\ast})\rangle^{1-\rho}}ds
+C\lambda^{p-1}\int_{0}^{\lambda^{p-1}}\langle x(s^{\ast})\rangle^{2\rho-1}ds\\
&\leq C\lambda^{p-1}\int_{0}^{\lambda^{p-1}}\frac{\lambda|\xi|\langle s^{\ast}\rangle^{1-\rho}(1+|x|+\lambda|\xi|)^{1-\rho}}{\langle x+\lambda s^{\ast}\xi\rangle^{1-\rho}}ds\\
&+C_{t_0}\lambda^{p-1}(1+|x|+\lambda|\xi|)^{2\rho-1}\int_{0}^{\lambda^{p-1}}\langle x+\lambda s^{\ast}\xi\rangle^{2\rho-1}ds\\
&\leq C_{t_0,a,K}\lambda^{p-\rho}\int_{0}^{\lambda^{p-1}}\frac{\lambda|\xi|}{\langle x+\lambda s^{\ast}\xi\rangle^{1-\rho}}ds\\
&+C_{t_0,a,K}\lambda^{p-1+2\rho-1}\int_{0}^{\lambda^{p-1}}\langle x+\lambda s^{\ast}\xi\rangle^{2\rho-1}ds.
\end{align*}
In the last inequality of the above, the change of a variable $s\mapsto\sigma=|x|- \lambda|\xi|s^{\ast}$ lets the integrals change to
\begin{align*}
\int_{|x|+\lambda{t_0}|\xi|}^{|x|+(\lambda^p-\lambda t_0)|\xi|}\langle\sigma\rangle^{\rho-1}d\sigma
+\lambda^{-1}\int_{|x|+\lambda{t_0}|\xi|}^{|x|+(\lambda^p-\lambda t_0)|\xi|}\langle\sigma\rangle^{2\rho-1}d\sigma\leq C_{a,t_0,K}(\lambda^{\rho}+\lambda^{2\rho-1})
\end{align*}
since $p<1$.
Thus one can see that
\begin{align*}
&\int_{0}^{\lambda^{p-1}}|s^{\ast}|\left|(\nabla_x\bm{a})^{T}(s,x(s^{\ast}))\right|\left(|\xi(s^{\ast})|+|\bm{a}(s,x(s^{\ast}))|\right)ds\\
&\leq C_{t_0,a,K}(\lambda^{p}+\lambda^{p+4\rho-3})
=C_{t_0,a,K}\lambda\left({\lambda^{p-1}}+\lambda^{p+4\rho-4}\right)
\leq \lambda \frac{t_0}{4a}
\end{align*}
when we take $p$ satisfying $0<p<\min(4({1-\rho}),1)$ and the last inequality is actually valid if $\lambda$ is sufficiently large.
Since $\int_0^{t_0}|\bm{a}(s,x(s^{\ast}))|ds$ can be estimated similarly,
we obtain that
\begin{align*}
|x(0)|
\geq\lambda t_0|\xi|-|x|-\lambda\frac{t_0}{4a}-\lambda\frac{t_0}{4a}-\lambda\frac{t_0}{4a}
=O(\lambda)\ \text{as}\ \lambda\to\infty
\end{align*}
for any $x\in K$ and $\xi\in\Gamma_a$.
Since $|\Lambda|\sim\lambda^{-2b}$,
$|x(0)|=O(\lambda^{M})$ holds for some $M>0$,
which concludes our desired assertion.
\end{proof}
\setcounter{section}{0}
\renewcommand{\thesection}{\Alph{section}}
\section{Appendix}
\begin{lem}\label{L1}
Assume $H(t)$ of \eqref{eq1} is $\frac{1}{2}(\nabla-\bm{a}(t,x))^2$ and $\bm{a}(t,x)$ satisfies Assumption~\ref{A2}.
Let $0<p<1$, $a\geq1$, $t_0\in\R$, $x_0, \xi_0\in\R^n$, $K$ be a (relatively compact) neighborhood of $x_0$ and $\Gamma$ be a conic neighborhood of $\xi_0$.
There exists a positive constant $\lambda_0$ such that
\begin{align}\label{a1}
\begin{cases}
\frac{1}{2a}\left|s^*\right|\lambda\leq\left|x(s^*)\right|\leq2a|s^*|\lambda\\
\frac{1}{2a}\lambda\leq\left|\xi(s^*)\right|\leq2a\lambda
\end{cases}
\end{align}
for all $\la\geq\la_0$, $\la^{p-1}\leq|s^*|\leq t_0$, $x\in K$ and $\xi\in\Gamma_a$.
\end{lem}
\begin{proof}
We adopt induction for the Picard approximation of solutions $x(s^{\ast})$ and $\xi(s^{\ast})$ of \eqref{ce} given as
\begin{align*}
\begin{cases}
x^{(1)}(s^{\ast})&=x+s^{\ast}\lambda\xi,\\
\xi^{(1)}(s^{\ast})&=\lambda\xi,\\
x^{(l)}(s^{\ast})&=x+\int_{t_0}^{s}(\nabla_{\xi}h)(\tau, x^{(l-1)}(\tau^{\ast}), \xi^{(l-1)}(\tau^{\ast}))d\tau,\quad l\geq2,\\
\xi^{(l)}(s^{\ast})&=\lambda\xi-\int_{t_0}^{s}(\nabla_{x}h)(\tau, x^{(l-1)}(\tau^{\ast}), \xi^{(l-1)}(\tau^{\ast}))d\tau, \quad l\geq2,
\end{cases}
\end{align*}
namely, we simply show \eqref{a1} for $x^{(l)}(s^{\ast})$ and $\xi^{(l)}(s^{\ast})$ for any $l\in\N$ instead of $x(s^{\ast})$ and $\xi(s^{\ast})$ by induction.
For $l=1$, \eqref{a1} holds obviously. 
For $l=2$, since $\rho<1$, $\lambda|s^{\ast}|\geq\lambda^p$ and $(\nabla_{\xi}h)(\tau,x^{(l)}(\tau^{\ast}),\xi^{(1)}(\tau^{\ast}))=\xi^{(l-1)}(\tau^{\ast})-\bm{a}(\tau,x^{(l-1)}(\tau^{\ast}))$, we have
\begin{align*}
|x^{(2)}(s^{\ast})|&\geq |x+s^{\ast}\lambda\xi|-\int_{t_0}^s|\bm{a}(\tau,x^{(1)}(\tau^{\ast}))|d\tau\\
&\geq \lambda|s^{\ast}||\xi|-|x|-\int_{t_0}^s\langle 2a\lambda|\tau^{\ast}|\rangle^{\rho} d\tau\\
&\geq \lambda|s^{\ast}||\xi|-|x|-(2a\lambda)^{\rho}\int_{t_0}^s\langle \tau^{\ast}\rangle^{\rho} d\tau\\
&\geq \frac{1}{2a}\lambda|s^{\ast}|\left(1-\frac{2a|x|}{\lambda^p}-(2a)^{1+\rho}\lambda^{p(\rho-1)}\right)
\geq \frac{1}{2a}\lambda|s^{\ast}|
\end{align*}
for $\lambda\geq\lambda_0$ with some sufficiently large $\lambda_0\geq1$.
Meanwhile, we estimate $\xi^{(2)}(s^{\ast})$ as 
\begin{align*}
|\xi^{(2)}(s^{\ast})|&\geq \left|\lambda\xi-\int_{t_0}^s(\nabla_x\bm{a})^T(\tau,x^{(1)}(\tau^{\ast}))(\lambda\xi-\bm{a}(\tau,x^{(1)}(\tau^{\ast})))d\tau\right|\\
&\geq \frac{1}{2a}\lambda\left(1-(2a)^2\int_{t_0}^s|(\nabla_x\bm{a})^T(\tau,x^{(1)}(\tau^{\ast}))|d\tau-\frac{2a}{\lambda}\int_{t_0}^s\langle2a\lambda|\tau^{\ast}|\rangle^{2\rho-1}d\tau\right)\\
&\geq \frac{1}{2a}\lambda\left(1-\frac{(2a)^{1+\rho}|s^{\ast}|}{\lambda^{1-\rho}}-\frac{(2a)^{2\rho}}{\lambda^{2(1-\rho)}}\langle s^{\ast}\rangle^{2\rho}\right)\geq \frac{1}{2a}\lambda
\end{align*}
for $\lambda\geq\lambda_0$ with some sufficiently large $\lambda_0\geq1$, since $-(\nabla_{x}h)(\tau,x^{(l)}(\tau^{\ast}),\xi^{(l)}(\tau^{\ast}))=(\nabla_x\bm{a})^T(\tau,x^{(l-1)}(\tau^{\ast}))(\xi^{(l-1)}(\tau^{\ast})-\bm{a}(\tau,x^{(l-1)}(\tau^{\ast})))$. We can show $|x^{(2)}(s^{\ast})|\leq 2a\lambda|s^{\ast}|$ and $|\xi^{(2)}(s^{\ast})|\leq 2a\lambda$ in the same way.

Finally, suppose that $x^{(l)}(s^{\ast})$ and $\xi^{(l)}(s^{\ast})$ satisfy \eqref{a1} for all $l\leq N$ with some $N\geq 2$.
Inspired by the above calculation, $x^{(N+1)}(s^{\ast})$ can be estimated as
\begin{align*}
&|x^{(N+1)}(s^{\ast})|\\
&=\left|x+s^{\ast}\lambda\xi\right|-\int_{t_0}^{s}|\bm{a}(\tau,x^{(N-1)}(\tau))|d\tau\\
&-\int_{t_0}^{s}|s-\tau||(\nabla_{x}\bm{a})^T(\tau, x^{(N-2)}(\tau))|\left(| \xi^{(N-2)}(\tau)|+|\bm{a}(\tau,x^{(N-2)}(\tau))|\right)d\tau\\
&\geq \frac{1}{2a}\lambda|s^{\ast}|\left(1-\frac{2a|x|}{\lambda^p}-(2a)^{1+\rho}\lambda^{p(\rho-1)}\right)
\\
&-\frac{1}{2a}\lambda\left(1-\frac{|s^{\ast}|}{(2a\lambda)^{1-\rho}}-\frac{(2a)^{\rho}}{\lambda^{2(1-\rho)}}\langle s^{\ast}\rangle^{2\rho}\right)\int_{t_0}^{s}|s-\tau|d\tau
\geq \frac{1}{2a}\lambda|s^{\ast}|
\end{align*}
for $\lambda\geq\lambda_0$ with some sufficiently large $\lambda_0\geq1$, which is independent of $N$. We can show $|x^{(N+1)}(s^{\ast})|\leq 2a\lambda|s^{\ast}|$ and $\frac{1}{2a}\lambda\leq|\xi^{(N+1)}(s^{\ast})|\leq 2a\lambda$ in the same way.

\end{proof}
\begin{lem}\label{L2}
Let $\delta>0$. Then there exists a constant $C_{\delta}>0$ such that for any $\tau, t\in\mathbb{R}$, $x,\xi\in\mathbb{R}^n$ and compact interval $I=[a,b]\subset \R$ $(-\infty<a<b<\infty)$, it holds that
\begin{align}\label{a2}
\int_I\frac{\left|\xi(\tau;t,x,\xi)\right|}{\langle x(\tau;t,x,\xi)\rangle^{1+\delta}}d\tau\leq C_{\delta}(1+|b-a|),
\end{align}
where $x(\tau)=x(\tau; t, x, \xi)$ and $\xi(\tau)=\xi(\tau; t, x, \xi)$ as the solutions of \eqref{ce}.
\end{lem}
The above estimate of characteristics \eqref{ce} plays a key role for the proof of our result.
The lemma is proved by putting $q(t)=x(t)$, $v(t)=\xi(t)-\bm{a}(t,x(t))$ in the proof of Lemma 2.1 in \cite{Y0}.

The following formulae are important for our proof and obtained immediately from the commutator relation $x^{\alpha}\partial_x^{\beta}e^{it\Delta/2}=e^{it\Delta/2}(x-it\nabla_x)^{\alpha}\partial_x^{\beta}$ and scaling argument (for the proof, see \cite{KI}).
\begin{lem}\label{L3}
Let $\varphi_0\in\Sh(\R^n)$, $s\in\R$, $\lambda\geq1$, $0<b\leq1$ and $\varphi_{0,\lambda}(x)=\lambda^{nb/2}\varphi_0(\lambda^bx)$.
Define $U_0(t)=e^{it\Delta/2}$ and $\varphi^{(s)}_{\lambda}(x)=U_0(s)[\varphi_{0,\lambda}](x)$.
Then, for any $\alpha, \beta\in\Z^n_{+}$, it holds that
\begin{equation}\label{a3}
x^{\alpha}\partial_x^{\beta}\varphi^{(s)}_{\lambda}(x)=\sum_{\substack{\gamma_1+\gamma_2=\alpha\\\gamma'_1\leq\gamma_1,\gamma'_2\leq\gamma_2}}s^{|\beta+\gamma|}\lambda^{b(|\beta|+|\gamma|-|\alpha|)}\left|(\varphi^{\beta,\alpha})^{(s)}_{\lambda}(x)\right|.
\end{equation}
As a consequence from the above equality, we also have the additional inequality
\begin{equation}\label{a4}
\left|(1+|x|^2)^M\varphi^{(s)}_{\lambda}(x)\right|\leq\sum_{|\beta+\gamma|\leq 2M}C_{\beta,\gamma}s^{|\gamma|}\lambda^{b(|\gamma|-|\beta|)}\left|(\varphi^{\gamma,\beta})^{(s)}_{\lambda}(x)\right|
\end{equation}
for arbitrary $M\in\N$.
\end{lem}

\end{document}